\documentclass[10pt]{amsart}
\usepackage{graphicx}
\usepackage{amssymb}
\usepackage{amsthm}
\usepackage{amscd}
\usepackage{multicol, verbatim}
\usepackage{amsmath, amsthm, mathrsfs, hyperref}
\hypersetup{linkcolor=black, citecolor=black, colorlinks=true, final=true, breaklinks, hyperindex=false,
pdfauthor={Matthew Chasse, Lukasz Grabarek, Mirko Visontai}}

\theoremstyle{plain}
   \newtheorem{theorem}{Theorem}[section]
   \newtheorem{proposition}[theorem]{Proposition}
   \newtheorem{lemma}[theorem]{Lemma}
   \newtheorem{corollary}[theorem]{Corollary}
   \newtheorem{problem}{Problem}
   \newtheorem{conjecture}[theorem]{Conjecture}
   \newtheorem{question}{Question}[section]
\theoremstyle{definition}
   
   \newtheorem{example}[theorem]{Example}

\theoremstyle{remark}
   \newtheorem{remark}[theorem]{Remark}

\def\RR{\mathbb{R}}

\def\CC{\mathbb{C}}

\def\NN{\mathbb{N}}

\def\PP{\mathcal{P}}
\def\BB{\mathcal{B}}

\def\LL{\mathscr{L}}
\def\TT{\mathscr{T}}

\newcommand{\bea}{\begin{eqnarray*}} %aligned equal signs
\newcommand{\eea}{\end{eqnarray*}}  %the * removes equation numbering
\newcommand{\bnum}{\begin{enumerate}}%shortcut for enumerated lists
\newcommand{\enum}{\end{enumerate}}
\newcommand{\bit}{\begin{itemize}}  %shortcut for itemized lists
\newcommand{\eit}{\end{itemize}}
\newcommand{\beq}{\begin{equation}}  %shortcut for equations
\newcommand{\eeq}{\end{equation}}

\def\today{\ifcase\month\or
       January\or February\or March\or April\or May\or June\or
       July\or August\or September\or October\or November\or December\fi
       \space\number\day, \number\year}

\date{\today}
\title[Stable regions of Tur\'an expressions]{Stable regions of Tur\'an expressions}

\author{Matthew Chasse}
\address{Department of Mathematics, Royal Institute of Technology, SE-100 44 Stockholm, Sweden}
\email{chasse@math.kth.se}

\author{Lukasz Grabarek}
\address{Department of Mathematics, University of Hawaii, Honolulu, HI 96822}
\email{lukasz@math.hawaii.edu}

\author{Mirk\'o Visontai}
\address{Department of Mathematics, Royal Institute of Technology, SE-100 44 Stockholm, Sweden}
\email{visontai@math.kth.se}


\begin{document} 

\maketitle

\begin{abstract}
Consider polynomial sequences that satisfy a first-order differential recurrence. 
We prove that if the recurrence is of a special form, then the Tur\'an expressions for the sequence are weakly Hurwitz stable (non-zero in the open right half-plane). 
A special case of our theorem settles a problem proposed by S.~Fisk that the Tur\'an expressions for the univariate Bell polynomials are weakly Hurwitz stable. 
We obtain related results for Chebyshev and Hermite polynomials, and propose several extensions involving Laguerre polynomials, Bessel polynomials, and Jensen polynomials associated to a class of 
real entire functions.   
\end{abstract}

\section{Introduction}

Given a sequence of polynomials $\PP=\{P_k\}_{k=0}^\infty$, with degree $\deg(P_k)=k$, we adopt the notation 
\begin{equation}\label{turan-ineq}
\TT_k(\PP;x) = (P_{k+1}(x))^2 - P_{k+2}(x)P_{k}(x)
\end{equation}
for the $k$-th \emph{Tur\'an expression}.

Let $D=\partial/\partial x$ and $\NN=\{0,1,2,\dots\}$ throughout.  A polynomial in $\CC[x]$ is called \emph{weakly Hurwitz stable} if it is 
non-zero in the right half-plane $\Re~x>0$.  Our investigation is motivated, in part, by a problem proposed by S.~Fisk.

\begin{problem}[{S.~Fisk \cite[p.~724, Question~33]{Fisk}}]\label{fisk-conjecture}
Let $\BB=\{B_k(x)\}_{k=0}^\infty$, where $B_k$ is the $k$-th univariate Bell polynomial given by the recurrence
\[
B_{k+1}(x) = x(B_{k}(x) + DB_{k}(x)),\qquad k\in\NN,
\]
with initial value $B_0(x) = 1$.  Then $\TT_k(\BB;x)$ is weakly Hurwitz stable for $k\in\NN$.
\end{problem}

Problem \ref{fisk-conjecture} concerns the univariate Bell polynomials, which have significance in combinatorics (\cite{Carlitz}, \cite{HHN}, \cite[p.~76]{Riordan}), and states that an 
associated Tur\'an expression is weakly Hurwitz stable, a property relevant in the analysis of linear systems \cite[p.~465]{AM}.  Unfortunately, Fisk passed away before he could finish his book, 
and the tentative proof of Problem \ref{fisk-conjecture} which appears in the electronic draft \cite[p.~651,~Lemma~21.91]{Fisk} is incorrect, as we show in Example \ref{coneless}  
(this may be why it also appears in the open question list \cite[p.~724, Question~33]{Fisk}).  By means of the following theorem we are able to answer Problem \ref{fisk-conjecture} in the affirmative.

\begin{theorem}\label{main-theorem}

Let $\PP=\{P_k\}_{k=0}^\infty$ be a set of real polynomials such that $\deg(P_k)=k$, for $k\in\NN$
\renewcommand\theenumi{\roman{enumi}}
\begin{enumerate}
\item \label{mt-1}
If the sequence $\PP$ satisfies
\[
P_{k+1}(x) = a(x+b)(D+c_{k})P_{k}(x),
\]
$a\neq 0$, $b\ge0$, and $c_{k+1}\ge c_{k}>0$ for all $k\in\NN$, 
then $\TT_k(\PP;x-b)$ is weakly Hurwitz stable for all $k\in\NN$.

\item \label{mt-3}
Let $\PP=\{P_k\}_{k=0}^\infty$ be a set of polynomials which satisfy
\[
P_{k+1}(x)=c(-ax + b + D)P_{k}(x)
\]
for $a>0,b,c\in\RR$, for each $k\in\NN$. 
For each $k \ge 1$, all $k$ zeros of $P_k$ are real (see Remark \ref{rem-real-roots}); let $M_k$ denote that largest and $m_k$ denote the smallest zero of $P_k$. 
Then for each $k\in\NN$, $\TT_k(\PP;x)$ is non-zero when $\Re~x>M_{k+1}$ or $\Re~x<m_{k+1}$. 

\item \label{mt-2}
If the sequence $\PP$ satisfies
\[(xD-k)P_k(x) = P_{k-1}(x)\]
for all $k\ge 1$, and each $P_k$ has only non-positive zeros, then $\TT_k(\PP;x)$ is weakly Hurwitz stable for all $k\in\NN$.

\end{enumerate}
\end{theorem}

\begin{remark}\label{rem-real-roots}
Note that the types of sequences in Theorem \ref{main-theorem} parts \eqref{mt-1} and \eqref{mt-3} consist of polynomials which have only real zeros.  This follows immediately from the recurrences, since the operators $a(x+b)(D+c_{k-1})$ and $c(-ax + b + D)$ preserve the property of having only real zeros (cf. \cite[p. 155]{RS}).  Although real, non-positive zeros are required in the hypothesis in part \eqref{mt-2}, the operator $(xD-k)$ preserves the property of having real zeros for polynomials of degree $k$ \cite[p. 97]{RS}.  In addition, in order to avoid a vacuous case for the theorem in parts \eqref{mt-1} and \eqref{mt-3} one should take $P_0$ to be a non-zero real constant. 
\end{remark}

Indeed, with $a =1$, $b=0$, and  $c_k =1$ for all $k\in\NN$, Theorem~\ref{main-theorem} \eqref{mt-1} establishes Problem \ref{fisk-conjecture}. 

%We also take a closer look at the univariate Bell polynomials in Section \ref{section-bell}, and exhibit some properties in addition to Proposition \ref{fisk-conjecture} which have not yet been 
%observed in the extensive literature on the Bell polynomials.  

\subsection{A Brief History of Tur\'an Expressions}
While obtaining a result concerning the root separation of consecutive Legendre polynomials, P.~Tur\'an \cite{T} proved the following theorem.

\begin{theorem}[\cite{T}]\label{turan-theorem}
Let $P_k$ be the $k$-th Legendre polynomial, and $\PP=\{P_k\}_{k=0}^\infty$ be the sequence of Legendre polynomials.  Then
\begin{equation}\label{TuransLegendre}
\TT_k(\PP;x) \ge 0, 
\end{equation}
for all $-1\le x\le 1$, and $k\in\NN$, where equality holds only for $|x|=1$.  
\end{theorem}

Tur\'an's paper appeared in print after a publication of G.~Szeg\H{o} \cite{Sz}, who provided four additional proofs of 
Theorem \ref{turan-theorem}.  Citing assistance from G.~P\'olya, Szeg\H{o}  \cite{Sz} establishes similar inequalities for the ultraspherical polynomials, the generalized Laguerre polynomials, 
and the Hermite polynomials.  Inequalities in the form \eqref{TuransLegendre} have been ardently investigated and extended, with contributions from a diverse assemblage of 
mathematicians \cite{CC,DS,KSz,P2,S2,S1,Sz,T}.  

For $P\in\RR[x]$ a fixed polynomial of degree exactly $n$, an important form of the Tur\'an expression \eqref{turan-ineq} is obtained for the sequence 
\[
\{P_{k}(x)\}_{k=0}^\infty=\{D^nP(x),D^{n-1}P(x),\dots,P(x),0,0,\dots\}
\]
 and appears in the {\em Laguerre inequality} \cite{CC}.  The Laguerre inequality and its analogs are useful in determining bounds for extreme zeros of a polynomial sequence \cite{FK2}, and in the study of the location of zeros of entire functions \cite{CC}.  
Tur\'an expressions are also ubiquitous in areas such as matrix theory \cite{CCTP} and statistics \cite{KSz}.

The Tur\'an expression can be represented in terms of a $2\times 2$ Hankel determinant 
\[
\TT_k(\PP;x)=-\begin{vmatrix}
P_{k}(x) & P_{k+1}(x) \\
P_{k+1}(x) & P_{k+2}(x)
\end{vmatrix},
\]
which invites extension to higher dimensional determinants.  In a comprehensive paper, S.~Karlin and G.~Szeg\H{o} \cite{KSz} addressed the properties of the higher dimensional 
determinants in the case that $x$ is real (see also \cite{Leclerc}).
For a special case of the Tur\'an expression, K.~Dilcher and K.~Stolarsky \cite{DS} proved the 
following theorem regarding the location of zeros. 
\begin{theorem}[\cite{DS}]\label{ds-wronskian}
If all the zeros of $P(z)$ are real and lie in the interval $[-1,1]$, then all the zeros of the Wronskian $W_P(x)=(P'(x))^2 - P''(x)P(x)$ lie inside or on the unit circle.
\end{theorem}
\noindent
Dilcher and Stolarsky take their analysis well beyond Theorem \ref{ds-wronskian}, and extend their results to generalized Laguerre expressions (see \cite{DS2}, Theorem \ref{ds-extended-theorem}).  
One such generalization gives rise to the \emph{extended Laguerre inequalities}, which have importance in the theory of entire functions.  In particular, these inequalities provide a necessary 
and sufficient condition for the Riemann-Hypothesis \cite{CV}.  We investigate extended Tur\'an expressions, and show that the extended Tur\'an expressions for the Chebyshev polynomials of the 
first and second kind have zeros only at $\pm 1$. 

\medskip
We note that our method of proof for Theorem \ref{main-theorem} yields a less precise version of Theorem \ref{ds-wronskian}, that the zeros of $W_P$ must lie in the vertical 
strip $-1\le \Re~x \le 1$.  The proof of Theorem \ref{main-theorem} uses standard methods involving the logarithmic derivative which are common in the analytic theory of polynomials \cite{RS}, and 
is similar to the proof of Theorem \ref{ds-wronskian}.  Results related to log-concavity have recently been obtained by D.~Karp \cite{Karp}, who also conjectures Hurwitz stability for polynomials which arise from determinants. 

The rest of the paper is organized as follows. An example in Section \ref{prelim} shows that the drafted proof of Problem \ref{fisk-conjecture} contains a mistake (see \cite[p.~651,~Lemma~21.91]{Fisk}), and we note an observation 
(Lemma \ref{arg-arg}) 
for later use.  
In Section \ref{section-turan}, we establish the weak Hurwitz stability of the Tur\'an expressions for several forms of polynomial sequences (Propositions \ref{lem:main}, \ref{hermite-type}, 
and \ref{polar-type}), which proves Theorem \ref{main-theorem}.  We also obtain an analogous result for the Hermite polynomials (Proposition \ref{hermite-turan}).  
An \textit{extended Tur\'an expression} is introduced in Section \ref{high-turan}, along with some supporting results for the Chebyshev polynomials 
(Propositions \ref{prop-high-cheby} and \ref{prop-high-cheby2}), and some conjectures regarding the Laguerre and single variable Bell polynomials (Conjectures \ref{bell-conjecture} and 
\ref{laguerre-conjecture}).

\section{Preliminary Observations}
\label{prelim}

We begin by recalling two results needed for our investigation.

\begin{proposition}[{cf. \cite[p. 366, Proposition 11.4.2]{RS}}]\label{weak-stable-necessary}
If $f\in\RR[x]$ is weakly Hurwitz stable, then the coefficients of $f$ are non-negative. 
\end{proposition}
 
\begin{proof}
This follows immediately from the non-negativity of the coefficients of the irreducible linear and quadratic factors of $f$.
\end{proof}

For a complex polynomial $f(x)=\sum_{k=0}^n (a_k+ib_k)x^k$, where $a_k,b_k\in\RR$ for $k=0,\dots,n$, let the real and imaginary parts of $f$ be given by $\Re f(x)=\sum_{k=0}^n a_kx^k$, and $\Im f(x)=\sum_{k=0}^n b_kx^k$, respectively.
Given a sequence of real numbers $\{r_1,\dots,r_n\}$, let $V(r_1,\dots,r_n)$ be the number of sign changes of the sequence, where $0$ values are omitted (e.g. $V(+,0,-,+,0)=V(+,-,+)=2$).  
Given two polynomials $f_0$ and $f_1$, with $\deg(f_1) <\deg(f_0)$, we generate a sequence by the standard Euclidean algorithm with the following recursion: given $\deg(f_{k-1})>\deg(f_k)$, determine the 
quotient $q_k$ and remainder $r_k$ such that $f_{k-1}=q_kf_{k}+r_k$ and $\deg(r_k)<\deg(f_k)$, then set $f_{k+1}=-r_k$, and terminate when $f_{k+1}$ is the greatest common divisor of $f_0$ and $f_1$.

\begin{theorem}[{\cite[p. 364, Proposition 11.3.2]{RS}}]\label{sturm-hurwitz-criterion}
Let $f$ be a monic polynomial of degree $n$.  Denote by $f_0,f_1,\dots,f_m$ the sequence of polynomials generated by the Euclidean algorithm, started with $f_0(x)= \Re f(x)$ and $f_1(x)=\Im f(x)$, 
and set $m=0$ if $f_1\equiv 0$. If $f$ has no real zeros, then the number $p$ of zeros of $f$ in the open upper half-plane is given by
\[ 
p = \frac{n}{2} + \lim_{R\to\infty} \left[ V(f_0(R),\dots,f_m(R)) - V(f_0(-R),\dots,f_m(-R)) \right]
\]
\end{theorem}

\begin{remark}\label{zero-count}
Let $g$ denote the greatest common divisor of $f_0\in\RR[x]$ and $f_1\in\RR[x]$ associated to a monic polynomial $f=f_0+if_1$. Since $g$ is the factor of $f$ containing all its real zeros, 
applying Theorem \ref{sturm-hurwitz-criterion} to $f/g$ counts the number of zeros $p$ of $f$ in the open upper half-plane. Then $m=p+ \deg(g)$ is the number of zeros of $f$ in the \emph{closed} 
upper half-plane.  
\end{remark}

Theorem \ref{sturm-hurwitz-criterion} and Remark \ref{zero-count} permit determination of the weak Hurwitz stability of any $f\in\RR[x]$ with degree $n$ by verifying that all zeros of $(-i)^nf(ix)$ 
lie in the closed upper half-plane.  A frequently cited, equivalent result is the Routh--Hurwitz criterion \cite[p. 369, Proposition 11.4.5]{RS}, but we found it was slower for our computations.  

The proof of Theorem \ref{main-theorem} involves the following elementary observation, which was also noted by Fisk \cite[p. 627, Lemma 21.29]{Fisk}. Let $\arg:\CC\setminus\{0\}\to(-\pi,\pi]$ denote the conventional argument function.  

\begin{lemma}\label{arg-arg}
If $a,b,c\in\RR$, $c>0$, $x\in\CC$, and $\arg(x-a)\in (0,\pi/2)$, $\arg(x-b)\in [0,\pi/2)$, then
\renewcommand\theenumi{\roman{enumi}}
\begin{enumerate}
\item $\Im\left[\frac{c}{x-a}\right] < 0$, and
\item $\Im\left[\frac{c}{(x-a)(x-b)}\right] < 0$.
\end{enumerate}
\end{lemma}

\begin{proof}

Follows immediately: 
\[\arg(x-a)\in (0,\pi/2) \Rightarrow \arg(1/(x-a))\in(-\pi/2, 0) \Rightarrow (i)\, ,\] and  \[\arg(x-b)\in [0,\pi/2) \Rightarrow \arg(1/(x-b))\in(-\pi/2, 0]\, ,\] therefore 
\[\arg(1/(x-a)(x-b))\in(-\pi/0) \Rightarrow (ii).\]  
\end{proof}
 
The proof of Problem \ref{fisk-conjecture} that appears in Fisk's electronic draft \cite[p.~651, Lemma 21.91]{Fisk} relies on a claim that for a polynomial $f$ with only real negative 
zeros, the polynomial 
\begin{equation}\label{not-a-cone}
af^2+bff'+cx(ff''-(f')^2)
\end{equation}
is (weakly) Hurwitz stable for all $a,b,c>0$ \cite[p.~643, Lemma 21.67]{Fisk}.  The following example shows that this claim is false; in the draft, it is mistakenly assumed that the terms of the 
sum \eqref{not-a-cone} are part of a cone of ``half-plane interlacings'', although if $c$ were a negative number this would hold.

\begin{example}\label{coneless}
Let $f=x(x+2)$, $a=b=1$, and $c=3$.  The polynomial obtained in \eqref{not-a-cone} is then $-8x - 2 x^2 + x^4$, which is not weakly Hurwitz stable by Proposition \ref{weak-stable-necessary}.
\end{example}

\section{Stable Tur\'an  Expressions}\label{section-turan}

In this section we establish the stability of the Tur\'an expression for sequences of polynomials which satisfy first order recurrences of a special form 
(Proposition \ref{lem:main}, \ref{hermite-type}, and \ref{polar-type}), which proves Theorem \ref{main-theorem}.  

The zero sets of two polynomials $f,g\in\RR[x]$ are said to \emph{strictly interlace}, 
if $f$ and $g$ have only real zeros, and if between any two consecutive zeros of $f$ there is exactly one zero of $g$ and \textit{vice versa}.  
It is known that if the zeros of $f$ and $g$ are non-positive and interlace, then the difference in logarithmic derivatives $f'/f-g'/g$ is non-zero on the positive real axis.  
We use a difference of logarithmic derivatives in our proofs as well, but weak Hurwitz stability of the Tur\'an expression does not follow from the interlacing of the zeros alone 
(consider $f_0=1$, $f_1=x-1$, $f_2=x^2-3x/2$).

\begin{proposition}\label{lem:main}
Let $\PP=\{P_k\}_{k=0}^\infty$ be a set of polynomials where $\deg(P_k)=k$, 
\begin{equation} \label{rec1}
P_{k+1}(x) = a(x+b)(D+c_{k})P_{k}(x),
\end{equation}
$a \in \RR \setminus \{0\}$, $b\ge0$, and $c_{k+1}\ge c_{k}>0$ for all $k\in\NN$. 
Then $\TT_k(\PP;x)$ is non-zero for $\Re~x>-b$, $k\in\NN$.
\end{proposition}

\begin{proof}
We may assume that $P_0(x)=1$, $a=1$, by normalizing the sequence, and in this case, 
$P_1(x)=c_0(x+b)$ and $P_2(x)=c_0c_1(x+b)(x+b+1/c_1)$.  
Recall from Remark \ref{rem-real-roots} that the form of the recurrence ensures all the roots of $P_k(x)$ are real for each $k\in\NN$. 
Let $\{r_j\}_{j=1}^k$, $\{s_j\}_{j=1}^{k-1}$ be the zeros of $P_k$ and $P_{k-1}$ respectively, listed in decreasing order.  
We first show that $r_k<s_{k-1}<r_{k-1}<\cdots<s_1=r_1=-b \leq 0$. 

We prove the claim by induction.  Assume that $P_{k-1}(x)=(x+b)Q_{k-1}(x)$, where $Q_{k-1}(x)$ is a degree $k-2$ polynomial 
with only simple real zeros, strictly less than $-b$.  Now, let $c_{k-1}>0$.  From \eqref{rec1},

\[Q_{k}(x)=(D+c_{k-1})(x+b)Q_{k-1}(x)=e^{-c_{k-1} x}D\left(e^{c_{k-1} x}(x+b)Q_{k-1}(x)\right).\]
By an application of Rolle's theorem, $Q_k(x)$ has only real simple zeros strictly less than $-b$, which strictly interlace those of $Q_{k-1}$.  
(Note the argument just given is based on a proof of the Hermite--Poulain Theorem \cite[p. 155]{RS}).

Observe that whenever $x$ is not a zero of $P_k(x)$, 
\[
\frac{P_{k+1}(x)}{P_k(x)} = a(x+b)\left(c_k + \frac{P'_k(x)}{P_k(x)}\right).
\]
Substituting this into the expression for $\TT_{k-1}$ yields (for all $x$ with $\Re~x> - b$)
\[
\TT_{k-1}(\PP;x) = (P_k(x))^2 \left(1 - \frac{c_{k} + P'_{k}(x)/P_k(x)}{c_{k-1} + P'_{k-1}(x)/P_{k-1}(x)}\right).
\]
The expression $\TT_{k-1}(\PP;x)$ is zero at $x=x_0$ if and only if 
\[
c_k + P'_{k}(x_0)/P_k(x_0)=c_{k-1} + P'_{k-1}(x_0)/P_{k-1}(x_0),
\]
or $P_k(x_0)=P_{k+1}(x_0)=0$, or $P_k(x_0)=P_{k-1}(x_0)=0$ --- by construction, the last two conditions occur only if $x_0=-b$.  To prove the proposition, we will show that for $\Re~x_0> -b$,
\begin{equation}\label{log-derivatives}
%W[P_{k-1},P_k]/(P_{k-1}P_k) = 
P'_{k}(x_0)/P_k(x_0)-P'_{k-1}(x_0)/P_{k-1}(x_0) +(c_k-c_{k-1})\neq 0.
\end{equation}
The first two terms of the left-hand side of \eqref{log-derivatives} can be written
\begin{align}
 P'_{k}(x_0)/P_k(x_0)-P'_{k-1}(x_0)/P_{k-1}(x_0) &= \sum_{j=1}^k \frac{1}{x_0-r_j} - \sum_{j=1}^{k-1} \frac{1}{x_0-s_j} \nonumber\\
&= \frac{1}{x_0-r_k} + \sum_{j=2}^{k-1}\frac{r_j-s_j}{(x_0-r_j)(x_0-s_j)},\label{log-sum}
\end{align}
where $r_j-s_{j}>0$, for $j=2,3,\ldots,k-1.$  

Suppose that  $\Re~x_0> -b$.  If $x_0$ is on the real axis then all of the terms in \eqref{log-sum} are positive and \eqref{log-derivatives} holds.  Let $\Im~x_0> 0$.  
By Lemma \ref{arg-arg}, all terms in \eqref{log-sum} have negative imaginary part.
Thus the imaginary part of \eqref{log-sum} must be negative when $\Im~x_0> 0$, and \eqref{log-derivatives} holds again.  The case $\Im~x_0< 0$ follows by symmetry.
\end{proof}

Proposition~\ref{appell-type} implies the Laguerre expression is non-zero outside the strip containing the zeros of the polynomial, although in this case a more precise result has already been 
proved by Dilcher and Stolarsky \cite{DS2,DS}.  

\begin{proposition}[{Corollary of \cite[Theorem 2.5]{DS}}]\label{appell-type}
Let $\PP=\{P_k\}_{k=0}^\infty$ be a set of polynomials where $P_k(x) = P'_{k+1}(x)$, and $P_k$ has only real non-positive zeros for $k=0,\dots,N.$ Then $\TT_k(\PP;x)$ is weakly Hurwitz 
stable for each $k=0,\dots,N.$
\end{proposition}

\begin{proof}[Proof (Sketch)]
Note that for $\Re~x_0>0$,
\[
\frac{P_{k-1}(x_0)}{P_{k}(x_0)} =  \frac{P'_k(x_0)}{P_{k}(x_0)},
\]
and
\[
T(\PP;x_0) = (P_k(x_0))^2 \left(1 - \frac{P'_{k}(x_0)/P_k(x_0)}{P'_{k+1}(x_0)/P_{k+1}(x_0)}\right).
\]
Similar to Proposition \ref{lem:main} we need to show that 
\begin{equation}\label{apeq} 
P'_{k}(x_0)/P_k(x_0)-P'_{k-1}(x_0)/P_{k-1}(x_0)\neq 0.
\end{equation}
Assume first that $P_k(x)$ has only simple real zeros $\{r_j\}_{j=1}^k$.  Then the zeros $\{s_j\}_{j=1}^{k-1}$ of $P'_k(x)=P_{k-1}(x)$ are simple and interlace those of $P_k(x)$ ($r_k<s_{k-1}<\cdots<s_1<r_1$). 
The proof that \eqref{apeq} holds when $\Re~x>0$ then proceeds in the same manner as in Proposition \ref{lem:main}.  The claim follows for polynomials with multiple real zeros by passing to the limit.
\end{proof}

A statement similar to Proposition \ref{appell-type} holds for a type of recurrence which is satisfied by the Hermite polynomials.

\begin{proposition}\label{hermite-type}
Let $\PP=\{P_k\}_{k=0}^\infty$ be a set of polynomials which satisfy
\begin{equation}\label{hrec}
P_{k+1}(x)=c(-ax + b + D)P_{k}(x)
\end{equation}
for $a>0,b,c\in\RR$, for each $k\in\NN$. 
For each $k \ge 1$, all $k$ zeros of $P_k$ are real (see Remark \ref{rem-real-roots}); let $M_k$ denote that largest and $m_k$ denote the smallest zero of $P_k$.  
Then for each $k\in\NN$, $\TT_k(\PP;x)$ is non-zero when $\Re~x>M_{k+1}$ or $\Re~x<m_{k+1}$. 
\end{proposition}

\begin{proof}
We show that $\TT_{k-1}(\PP;x)$ is non-zero when $\Re~x_0>M_k$ or $\Re~x_0<m_{k}$ for $k\ge 1$.
Fix $k\ge 1$ and suppose $\Re~x_0>M_k$. Then
\begin{equation}\label{aheq} 
\frac{P_{k+1}(x_0)}{P_k(x_0)}-\frac{P_k(x_0)}{P_{k-1}(x_0)}\neq 0
\end{equation}
implies $\TT_{k-1}(x_0)\neq 0$ for $x_0$ which is not a zero of $P_k$ or $P_{k-1}$.  
From Rolle's theorem and the relation
\[
P_{k}(x)= ce^{ax^2/2 -bx}De^{-ax^2/2 +bx}P_{k-1}(x),
\] 
the zeros of consecutive polynomials $P_{k-1}$ and $P_{k}$ are real and strictly interlace.  Therefore, establishing \eqref{aheq} is sufficient to complete the proof, since the interval defined by 
the extremal zeros of $P_k$ contains the extremal zeros of $P_{k-1}$.  
With \eqref{hrec}, \eqref{aheq} is equivalent to
\[
\frac{P'_{k}(x_0)}{P_k(x_0)}-\frac{P'_{k-1}(x_0)}{P_{k-1}(x_0)}\neq 0.
\]
Denote the zeros of $P_k$ and $P_{k-1}$ by $\{r_j\}_{j=1}^k$ and $\{s_j\}_{j=1}^{k-1}$ respectively, ordered such that $r_1>s_1>\dots>s_{k-1}>r_k$.  Then
\begin{equation}
\frac{P'_{k}(x_0)}{P_k(x_0)}-\frac{P'_{k-1}(x_0)}{P_{k-1}(x_0)} = \frac{1}{x_0-r_k} + \sum_{j=1}^{k-1} \frac{r_j-s_j}{(x_0-r_j)(x_0-s_j)},
\end{equation}
and by Lemma \ref{arg-arg} the expression on the right is non-zero when $\Im~x_0>0$, since each term has argument in $(0,\pi]$.  
When $x_0$ is real, $\Re~x_0>M_k$ implies that all the terms in the sum are positive.  Thus, \eqref{aheq} holds when $\Re~x_0>M_{k}$.  With 
\[
\frac{P'_{k}(x_0)}{P_k(x_0)}-\frac{P'_{k-1}(x_0)}{P_{k-1}(x_0)} = \frac{1}{x_0-r_1} + \sum_{j=1}^{k-1} \frac{r_{j+1}-s_j}{(x_0-r_{j+1})(x_0-s_j)} \, ,
\]
the same argument, \textit{mutatis mutandis}, shows that \eqref{aheq} holds when $\Re~x_0<m_{k}$.  
\end{proof}

\begin{corollary}\label{hermite-turan}
Let $\PP=\{H_k\}_{k=0}^\infty$ be the set of Hermite polynomials \cite[p.~187]{R}, and let $M_k$ be the maximal zero of $H_k$, $k\in\NN.$ Then, for each $k\in\NN$, $\TT_k(\PP;x)$ is 
non-zero when $|\Re~x|>M_{k+1}.$ 
\end{corollary}

\begin{proof}
The Rodrigues-type formula for the Hermite polynomials \cite[p.~189]{R} is 
\[
H_k(x) = (-1)^ke^{x^2}D^ke^{-x^2} \, ,
\]
whence
\begin{equation}\label{hermiteRec}
H_k(x)= -e^{x^2}De^{-x^2}H_{k-1}(x) = (2x-D)H_{k-1}(x).
\end{equation}
Now, because Hermite polynomials of even degree are even functions and Hermite polynomials of odd degree are odd functions, we may replace the condition on $x$ in Proposition \ref{hermite-type} with 
the requirement that $|\Re~x| > M_{k+1}$.  The corollary then follows from \eqref{hermiteRec} and Proposition \ref{hermite-type}.
\end{proof}

We note one more condition which is sufficient for a sequence of polynomials to have a stable Tur\'an expression.  We thank the anonymous referee for informing us that such sequences have appeared in \cite{Schur}.

\begin{proposition}\label{polar-type}
Let $\PP=\{P_k\}_{k=0}^\infty$ be a sequence of polynomials with only real non-positive zeros, such that $\deg(P_k)=k$.  If each $P_k\in\PP$ satisfies $(xD-k)P_k(x) = P_{k-1}(x)$ for $k\ge1$, then $\TT_k(\PP;x)$ is weakly Hurwitz stable.
\end{proposition}

\begin{proof}
Since
\[
P_{k-1}(x) = (xD-k)P_k(x) = x^{k+1}Dx^{-k}P_k(x),
\]
the zeros of $P_{k-1}$ interlace those of $P_k$  by Rolle's theorem.
The Tur\'an expression $\TT_{k-1}(\PP;x) = P_k^2(x)- P_{k+1}(x)P_{k-1}(x)$ vanishes at $x_0$ if and only if the difference $P_k(x_0)/P_{k+1}(x_0) - P_{k-1}(x_0)/P_k(x_0)$ also vanishes, provided $x_0$ is 
not a zero of $P_{k}$ or $P_{k+1}$. 
Using $(xD-k)P_k(x) = P_{k-1}(x)$, one deduces that $(P_{k-1}/P_k - P_k/P_{k+1})(x_0)\neq 0$ for $\Re~x_0> 0$ if and only if for $\Re~x_0> 0$,
\[
\frac{P'_{k+1}(x_0)}{P_{k+1}(x_0)} - \frac{P'_{k}(x_0)}{P_{k}(x_0)} - \frac{1}{x_0} \neq 0 \; .
\]
  Let $\{r_j\}_{j=1}^{k+1}$ be the zeros of $P_{k+1}$ and $\{s_j\}_{j=1}^k$ be the zeros of $P_{k}$, ordered such that
\[
r_{k+1}\le s_k \le r_k \le \cdots \le s_1 \le r_1<0.
\] 
Then, 
\begin{align}
\frac{P'_{k+1}(x_0)}{P_{k+1}(x_0)} - \frac{P'_{k}(x_0)}{P_{k}(x_0)} - \frac{1}{x_0} &= \frac{1}{x_0-r_1} -\frac{1}{x_0} + \sum_{j=2}^{k+1} \frac{1}{x_0-r_j} - \frac{1}{x_0-s_{j-1}}\nonumber\\
&=   \frac{r_1}{x_0(x_0-r_1)} + \sum_{j=2}^{k+1} \frac{r_j-s_{j-1}}{(x_0-r_j)(x_0-s_{j-1})}. \label{redundant}
\end{align}
Since every numerator in \eqref{redundant} has the same sign, for real $x_0\ge 0$, the sum must be non-zero.  Furthermore, if $\Im~x_0>0$, then by Lemma \ref{arg-arg}, each term in \eqref{redundant} 
has argument in $(0,\pi]$ and thus the sum of the terms must be non-zero.  The case $\Im~x_0< 0$ follows by symmetry.
\end{proof}

\begin{proof}[Proof of Theorem \ref{main-theorem}]
Statements \eqref{mt-1}, \eqref{mt-3}, and \eqref{mt-2} are given by Propositions \ref{lem:main}, \ref{hermite-type}, and \ref{polar-type}, respectively.
\end{proof}

\subsection{Tur\'an expressions for Laguerre and Jensen polynomials}

Given a real entire function $f(x)=\sum_{k=0}^\infty \gamma_k x^k/k!$, the {\em $n$-th Jensen polynomial associated with $f$} (cf. \cite{CC}) is 
\[
g_n(x)= \sum_{k=0}^n \binom{n}{k}\gamma_k x^k,
\]
where $n\in\NN$.
It is known that if $f$ is a locally uniform limit of polynomials whose zeros are all real, then all of its associated Jensen polynomials also have only real zeros.  Furthermore, $g_n(x/n)\to f$ locally uniformly \cite[Lemma 2.2]{CC}.  The relation $(n-xD)P_n(x) = nP_{n-1}(x)$ is satisfied when the sequence $\{P_n\}_{n=0}^\infty$ are Jensen polynomials associated to some entire function $f$ \cite[Proposition 2.1]{CC}, and also when it is the sequence of Laguerre polynomials \cite[p. 202]{R}. 

We have computed with \emph{Maxima} that the Tur\'an expression, $\TT_k$, is non-zero in the open left half-plane for the sequence of Laguerre polynomials up to $k=50$. 
Fisk suggests that furthermore the determinants of Hankel matrices of Laguerre polynomials are non-zero in the open left half-plane \cite[p. 653]{Fisk}.  This motivates the following question.

\begin{question}\label{q:jensen}
Let $\PP=\{P_k\}_{k=0}^\infty$ be a sequence of polynomials with only real non-positive zeros, such that $\deg(P_k)=k$.  Suppose that for positive integers $k$, $P_k\in\PP$ satisfies 
\[
(k-xD)P_k(x) = kP_{k-1}(x).
\]
Is $\TT_k(\PP;x)$ weakly Hurwitz stable for all $k\in\NN?$ 
\end{question}

Let $f(x)=\sum_{k=0}^\infty \gamma_k x^k/k!$, with $\gamma_k>0$ for all $k\in\NN$, be a real entire function that is a locally uniform limit of polynomials with only real zeros. T.~Craven and 
G.~Csordas have shown that the Jensen polynomials of $f$ satisfy $\TT_k(\{g_n\}_{n=0}^\infty;x)\ge 0$ for $x\ge0$ \cite[Theorem 2.3]{CC}.  If the answer to Question \ref{q:jensen} is affirmative, 
the Tur\'an expressions for the Laguerre polynomials (reflected across the origin), the Jensen polynomials associated with $f$, and for sequences $\{P_k\}_{k=0}^\infty$ generated by the relation $P_k=f(D)x^k (=x^kg_k(1/x))$, will 
all be weakly Hurwitz stable.    We make some stronger conjectures for the sequences of Bell and Laguerre polynomials in the next section 
(Conjecture \ref{bell-conjecture}, Conjecture \ref{laguerre-conjecture}). 

\section{Higher order Tur\'an expressions}\label{high-turan}

For a real entire function $f(x)$, the $n$-th \emph{extended Laguerre expression}, $\LL_n(f(x))$, is determined by
\[
|f(x+iy)|^2 = f(x+iy)f(x-iy) = \sum_{n=0}^\infty \LL_n(f(x))y^{2n}. 
\]
As remarked in \cite[p. 343]{CV} it follows that, 
\begin{equation}\label{extended-laguerre}
\LL_n(f(x)) = \sum_{j=0}^{2n} \frac{(-1)^{n+j}}{(2n)!}\binom{2n}{j}f^{(j)}(x)f^{(2n-j)}(x).
\end{equation}
Precise information on the locations of zeros for the $\LL_n(f(x))$ has been obtained by Dilcher and Stolarsky \cite{DS2}.  Note that one can also consider $\LL_n(f^{(k)}(x))$, which shifts the 
indices on the derivatives in \eqref{extended-laguerre} by $k$.  It is known that a real entire function $f$ can be obtained as a locally uniform limit of polynomials with only real zeros 
\emph{if and only if} $\LL_n(f(x))\ge 0$ for all $x\in\RR$, and for all $n=0,1,2,\dots$ \cite[p. 343]{CV}.  In analogy to the extended Laguerre expression, denote the $k$-th 
\emph{extended Tur\'an expression} for the sequence of polynomials $\PP = \{P_k(x)\}_{k=0}^\infty$ by
\begin{equation}\label{extended-turan}
\TT_k^{(n)}(\PP;x) = \sum_{j=0}^{2n} \frac{(-1)^{n+j}}{(2n)!}\binom{2n}{j}P_{j+k}(x)P_{2n+k-j}(x).
\end{equation}
When $n=1$ in \eqref{extended-turan}, the original Tur\'an expression is recovered, $\TT_k(\PP;x) = \TT_{k}^{(1)}(\PP;x)$.  

Let $G(x,t)=\sum_{n=0}^\infty P_n(x)t^n/n!$ be the exponential generating function for a sequence of polynomials $\PP=\{P_k\}_{k=0}^\infty$, and let $G^{(k)}(x,t) = (\partial/\partial t)^k G(x,t)$.
Then, similar to \eqref{extended-laguerre},
\begin{equation}\label{turan-gen}
G^{(k)}(x,it)G^{(k)}(x,-it) = \sum_{n=0}^\infty \TT^{(n)}_k(\PP;x)t^{2n}.
\end{equation}
The extended Tur\'an inequalities were investigated by M.~L.~Patrick \cite{P2}, who proved their positivity on prescribed regions of the real axis for the ultraspherical, generalized Laguerre, 
Hermite, and Chebyshev polynomials (first and second kind), and also for the sequences of their derivatives.

Let $\PP = \{T_k(x)\}_{k=0}^\infty$, where $T_k$ is the $k$-th Chebyshev polynomial of the first kind \cite[p. 301]{R}.  It is known that
\[ 
(T_k(x))^2 - T_{k+1}(x)T_{k-1}(x) = 1-x^2,
\]
for $k\in\NN$, and thus $\TT_k(\PP; x)$ always has zeros at $\pm1$.  This observation can be extended as follows.  

\begin{proposition}\label{prop-high-cheby}
If $\PP = \{T_k(x)\}_{k=0}^\infty$, where $T_k$ is the $k$-th Chebyshev polynomial of the first kind, then
\[
\TT_k^{(n)}(\PP;x) = \frac{2^{2n-1}}{(2n)!}(1-x^2)^n
\]
for $n\ge1$ and $k\in\NN.$
\end{proposition}

\begin{proof}
The exponential generating function for the Chebyshev polynomials (cf. \cite[p. 301]{R}) is
\begin{equation}\label{tc-gen}
G(x,t) = e^{xt}\cosh(t\sqrt{x^2-1}) = \sum_{n=0}^\infty \frac{T_n(x)}{n!}t^n.
\end{equation}
Re-writing \eqref{tc-gen} with exponential functions, one obtains
\[
2G^{(k)}(x,t)=(x+\sqrt{x^2-1})^ke^{t(x+\sqrt{x^2-1})}+(x-\sqrt{x^2-1})^ke^{t(x-\sqrt{x^2-1})}.
\]
Consequently,
\begin{align}
G^{(k)}(x,it)G^{(k)}(x,-it) &=  \frac{1}{4}\left((x+\sqrt{x^2-1})^{2k} + (x-\sqrt{x^2-1})^{2k}\right)\nonumber\\ 
& \qquad +\frac{1}{4}\left(e^{i2t\sqrt{x^2-1}} +e^{-i2t\sqrt{x^2-1}}\right)\nonumber\\
&= \frac{1}{2}T_{2k}(x) + \frac{1}{2}\cos\left(2t\sqrt{x^2-1}\right) \label{cheby-turan}.
\end{align}
Using \eqref{turan-gen}, the claim then follows from \eqref{cheby-turan} by inspecting the Maclaurin expansion of $\cos\left(2t\sqrt{x^2-1}\right)$ in $t$.  
\end{proof}

From the generating function,
\[
\sum_{n=0}^\infty \frac{U_n(x)}{n!}t^n = e^{xt}\left(\frac{x}{\sqrt{x^2-1}}\sinh(t\sqrt{x^2-1})+\cosh(t\sqrt{x^2-1})\right) \, ,
\]
a formula for the extended Tur\'an expressions can also be obtained for the Chebyshev polynomials of the second kind \cite[p. 301]{R}, $\{U_n\}_{n=0}^\infty$, using the same method of proof as in 
Proposition \ref{prop-high-cheby}. % (in this case $A_k^2+B_k^2=(-1)^{k+1}/(x^2-1)$) 

\begin{proposition}\label{prop-high-cheby2}
If $\PP = \{U_k(x)\}_{k=0}^\infty$, where $U_k$ is the $k$-th Chebyshev polynomial of the second kind, then
\[
\TT^{(n)}_k(\PP;x) = \frac{2^{2n-1}}{(2n)!}(1-x^2)^{n-1}
\]
for $n\ge1$ and $k,n\in\NN$.
\end{proposition}

We were unable to find similar formulas for the other orthogonal polynomial families, but the calculation for the extended Tur\'an expressions for the Chebyshev polynomials suggests the 
following conjectures.   

\begin{conjecture}\label{bell-conjecture}
If $\BB = \{B_k(x)\}_{k=0}^\infty$, where $B_k$ is the $k$-th univariate Bell polynomial, then $\TT^{(n)}_k(\BB;x)$ is weakly Hurwitz stable for $n\ge1$ and $k,n\in\NN$.
\end{conjecture}

\begin{conjecture}\label{laguerre-conjecture}
If $\PP = \{L_k(x)\}_{k=0}^\infty$, where $L_k$ is the $k$-th Laguerre polynomial, then $\TT_k^{(n)}(\PP;-x)$ is weakly Hurwitz stable for $n\ge1$ and $k,n\in\NN$.
\end{conjecture}

\begin{conjecture}\label{bessel-conjecture}
If $\PP = \{Y_k(x)\}_{k=0}^\infty$, where 
\[
Y_k(x) = \sum_{j=0}^k \frac{(k+j)!}{2^jj!(k-j)!}x^j
\]
is the $k$-th Bessel polynomial, then $\TT_k^{(n)}(\PP;x)$ is weakly Hurwitz stable for $n\ge1$ and $k,n\in\NN$.
\end{conjecture}

Conjectures \ref{bell-conjecture}, \ref{laguerre-conjecture}, and \ref{bessel-conjecture} have been verified for $1\le k \le 50$, and $1\le n\le 5$ using 
Theorem \ref{sturm-hurwitz-criterion} and Remark \ref{zero-count} implemented in Maxima. 

For the extended Laguerre expressions, Dilcher and Stolarsky have proved the following.

\begin{theorem}[{A special case of \cite[Proposition 3.1]{DS2}}]\label{ds-extended-theorem}
If all the zeros of the degree $d$ polynomial $P$ are located in the interval $[-1,1]$, then all the zeros of $\TT^{(n)}_k(\PP;x)$ lie inside or on the unit 
circle for $0 \le k \le d$, $0 \le n \le d$, where 
\[
\PP = \{D^dP, D^{d-1}P, \dots, P, 0, \dots\}.
\]
\end{theorem} 
\noindent
It seems reasonable to suspect that when a sequence of polynomials has interlacing zeros and the zero locations are similar to a sequence of consecutive derivatives, then approximate versions of the Theorem \ref{ds-extended-theorem} hold.

\end{document}